\documentclass[%review,
hidelinks,onefignum,onetabnum]{siamart220329}

\usepackage{lipsum}
\usepackage{amsfonts}
\usepackage{graphicx}
\usepackage{epstopdf}
\usepackage{algorithmic}
\ifpdf
  \DeclareGraphicsExtensions{.eps,.pdf,.png,.jpg}
\else
  \DeclareGraphicsExtensions{.eps}
\fi

\newsiamremark{remark}{Remark}
\newsiamremark{hypothesis}{Hypothesis}
\crefname{hypothesis}{Hypothesis}{Hypotheses}
\newsiamthm{claim}{Claim}

\headers{Homogenization of high-contrast composites}{T. Dang, Y. Gorb, and S. Jim\'{e}nez Bola\~{n}os}
\title{Homogenization of high-contrast dielectric elastomer composites\thanks{Submitted to the editors DATE.
\funding{The work of the third author was supported  by NSF grant DMS-2110036.}}}

\author{
   Thuyen Dang\thanks{Department of Statistics/Committee on
    Computational and Applied Mathematics, University of Chicago, 5747
    S. Ellis Avenue, Chicago, Illinois 60637
    (\email{thuyend@uchicago.edu}).} 
  \and
  Yuliya Gorb\thanks{National Science Foundation, 2415 Eisenhower Avenue, Alexandria, Virginia 22314
    (\email{ygorb@nsf.gov}).}
    \and
  Silvia Jim\'{e}nez Bola\~{n}os\thanks{Department of
    Mathematics, Colgate University, 13 Oak Drive, Hamilton, New York 13346
    (\email{sjimenez@colgate.edu}).} 
}

\usepackage{amsopn}

\usepackage{dsfont,stmaryrd}
\usepackage{mathtools}
\usepackage[mathcal]{euscript}
\usepackage{bm}
\usepackage{import}
\usepackage{xifthen}

\reversemarginpar

\usepackage{enumitem}
\newcommand{\ZZ}{{\mathbb Z}}

\newcommand{\RR}{{\mathbb R}}

\newcommand{\DD}{{\mathbb{D}}}
\newcommand{\bfa}{\mathbf{a}}

\newcommand{\bfc}{\mathbf{c}}
\newcommand{\bfC}{\mathbf{E}}
\newcommand{\bfd}{\mathbf{d}}
\newcommand{\bfB}{\mathbf{B}}

\newcommand{\bfr}{\mathbf{r}}
\newcommand{\bfQ}{\mathbf{Q}}
\newcommand{\bfR}{\mathbf{R}}
\newcommand{\bfT}{\mathbf{T}}

\newcommand{\bfV}{\mathbf{V}}
\newcommand{\bfW}{\mathbf{W}}
\newcommand{\bfp}{\mathbf{p}}

\newcommand{\bg}{\mathbf{g}}

\newcommand{\nn}{\mathbf{n}}
\newcommand{\ee}{\mathrm{e}}

\newcommand{\kk}{\mathbf{k}}
\newcommand{\uu}{\mathbf{u}}

\newcommand{\vv}{\mathbf{v}}
\newcommand{\ww}{\mathbf{w}}

\renewcommand\vec[1]{\boldsymbol #1} %no need bm package

\newcommand{\per}{\mathrm{per}}%{\sharp}%
\newcommand{\sym}{\mathrm{sym}}

\newcommand{\hmeas}{\mathcal{H}^{d-1}}

\newcommand{\abs}[1]{\left\lvert #1\right\rvert}

\newcommand{\norm}[1]{\left\lVert #1\right\rVert}
\newcommand{\tscale}{\xrightharpoonup[]{~2~}}%{\overset{2}{\rightharpoonup}}
\newcommand{\cv}[1][]{%
\ifthenelse{\isempty{#1}}{\xrightarrow[\hphantom{~2~}]{}}{\xrightarrow[\hphantom{~2~}]{#1}}%
}
\newcommand{\wcv}[1][]{%
\ifthenelse{\isempty{#1}}{\xrightharpoonup[\hphantom{~2~}]{}}{\xrightharpoonup[\hphantom{~2~}]{#1}}%
}

\newcommand{\calB}{\mathcal{B}}

\newcommand{\calD}{\mathcal{D}}

\newcommand{\calL}{\mathcal{L}}

\newcommand{\fraM}{\mathfrak{M}}

\DeclareMathOperator{\di}{d\!}
\DeclareMathOperator{\e}{{e}}
\DeclareMathOperator{\Div}{{div}}

\def\XXint#1#2#3{{\setbox0=\hbox{$#1{#2#3}{\int}$ }
\vcenter{\hbox{$#2#3$ }}\kern-.6\wd0}}

\def\coloneqq{\doteq}

\DeclarePairedDelimiterX\Set[1]\{\}{%
  #1%
}
\begin{document}

\maketitle

\begin{abstract}
This paper focuses on the homogenization of high-contrast dielectric elastomer composites, which are materials that deform in response to electrical stimulation. The considered heterogeneous material, consisting of an ambient material with inserted particles, is described by a weakly coupled system of an electrostatic equation with an elastic equation enriched with electrostriction. It is assumed that particles gradually become rigid as the fine-scale parameter approaches zero. This study demonstrates that the effective response of this system entails a homogeneous dielectric elastomer, described by a system of PDEs whose equations are fully decoupled. The coefficients of the homogenized equations depend on various factors, including the geometry of the composite, the periodicity of the original microstructure, and the coefficients characterizing the initial heterogeneous material. In particular, these coefficients are significantly influenced by the high-contrast nature of the coefficients of the fine-scale problem.
\end{abstract}

% REQUIRED

% REQUIRED

%\tableofcontents{}

\section{Introduction}
\label{sec:introduction}

Recently, a large volume of research has focused on the development of {\bf electroactive polymers} exhibiting a unique property known as electrostriction, meaning that they can change shape in response to electrical stimulation. Electroactive polymers offer significant advantages over other active materials, such as electroactive ceramics and shape memory alloys, due to their large strain capabilities, quick response times, low density, and high resilience; see, e.g.,\cite{bauerDielectricElastomers2014}.

Depending on the activation mechanism, electroactive polymers can be categorized into several groups, and this study focuses mainly on {\bf dielectric elastomers} due to the wide range of their applications. More specifically, dielectric elastomers are being used in actuators, often referred to as artificial muscles; in power generation applications to generate energy by converting mechanical energy into electrical energy; in sensors for detecting force and pressure, leveraging their electromechanical properties; in soft robotics, where adaptability and compliance are crucial, in various
biomedical applications, including medical devices and implants; in micro-beam resonators and other small-scale devices that require precise control and actuation; in optical applications, such as deformable surfaces for optics and adaptive optics systems, due to their ability to change shape in response to electrical stimuli; in applications such as noise-canceling windows and active vibration control systems for structures, among many others \cite{suoTheoryDielectricElastomers2010}.

Despite these benefits, the performance of dielectric actuators at high voltages is often limited by dielectric breakdown or electromechanical instability, followed by dielectric breakdown \cite{hakimisiboniDielectricElastomerComposites2013}. An approach to overcoming these limitations is to incorporate fillers with varying elastic and electrical properties into a soft elastomeric host, creating {\bf dielectric elastomer composites}. In \cite{hakimisiboniDielectricElastomerComposites2013}, the authors explore the improved performance of these composites, in particular, considering the case of {\it rigid particles} and highlighting the need for further investigation in this area. This advantage of dielectric elastomer composites is also stressed in other studies of dielectric elastomer composites with rigid particles \cite{hakimisiboniDielectricElastomerComposites2014,luDielectricElastomerActuators2012,tianDielectricElastomerComposites2012}.

The above is the motivation for our current study, which is concerned with the rigorous periodic homogenization of dielectric elastomer composites. Periodic homogenization involves assumptions on scale separation and a periodic
microstructure of size $0< \varepsilon \ll 1$. The governing equations defined, on a heterogeneous bounded domain $\Omega$,
describing the dielectric elastomer under consideration, consist of a linear electrostatic (scalar) equation in the presence of a bounded free body charge, weakly coupled with an elastic (vectorial) equation with ``high-contrast'' elastic coefficients that, in addition, involves an {\it electrostriction} term. Here, {\it high-contrast}
means that the elastic properties of the material are vastly different (more specifically, the particles gradually become rigid as the fine-scale parameter $\varepsilon$ approaches zero), and {\it weak coupling} means that the elastic displacement does not enter the electrostatic equation.  The most comprehensive and well-detailed derivation of this model, which resulted in the form of \eqref{eq:p410} below, was presented in
\cite{francfortEnhancementElastodielectricsHomogenization2021}, where the free energy function is approximated in the limit of small deformations and moderate electric fields. In addition, a similar model was considered in \cite{tianDielectricElastomerComposites2012}.
However, to the best of our knowledge, our paper is the first devoted to the rigorous mathematical homogenization of a high-contrast dielectric elastomer composite with rigid particles. For modeling the rigid particulate phase, we use the simplest high-contrast relation, as given by \eqref{eq:5} below.

The main goal of this study is to determine the {\it macroscopic} or {\it effective} behavior of the considered periodic composite by developing the homogenization theory for this composite and deriving the approximation that corresponds to the case $\varepsilon = 0$. 

A brief overview of studies conducted on the mathematical theory of dielectric elastomers, both in general and focusing specifically on homogenization, is highlighted as follows. The equations characterizing the behavior of deformable dielectrics were pioneered by Toupin in \cite{toupinElasticDielectric1956}, see also  \cite{francfortEnhancementElastodielectricsHomogenization2021,tianDielectricElastomerComposites2012}.
The rigorous homogenization, which amounts to developing an asymptotic analysis of the limiting response of the given PDE system as $\varepsilon \to 0$, for the model with {\it bounded coefficients} was carried in
\cite{francfortEnhancementElastodielectricsHomogenization2021,tianDielectricElastomerComposites2012} and
references therein. Homogenization of a similar model for a magnetic suspension was conducted in 
\cite{dangHomogenizationNondiluteSuspension2021,dangGlobalGradientEstimate2022}. As for {\it high-contrast problems}, the rigorous theory of  homogenization for such problems was studied in e.g. 
\cite{amazianeHomogenizationQuasilinearElliptic2006,balHomogenizationHydrodynamicTransport2021,bellieudTorsionEffectsElastic2009,brianeHomogenizationNonuniformlyBounded2002,cherdantsevHighContrastHomogenisation2017,davoliHomogenizationHighcontrastMedia2023,gerard-varetHomogenizationStiffInclusions2021,gorbEffectiveConductivityDensely2004,khruslovHomogenizedModelsComposite1991,liptonBlochWavesHigh2022,marcelliniHomogenizationNonuniformlyElliptic1978,marchenkoHomogenizationPartialDifferential2006,moscoCompositeMediaAsymptotic1994,shenLargescaleLipschitzEstimates2021,peterDifferentChoicesScaling2008,kamotskiTwoscaleHomogenizationGeneral2019,cherednichenkoNonlocalHomogenizedLimits2006}.
This paper contributes to the above list by extending the homogenization result found in
\cite{francfortEnhancementElastodielectricsHomogenization2021,tianDielectricElastomerComposites2012} for a \textbf{dielectic elastomer composite} to the \textit{high-contrast case}. The homogenization of a system of PDEs describing an elastic material with high-contrast {\it soft} elastic particles was carried out in \cite{cherdantsevHighContrastHomogenisation2017}; in contrast, our results tackle an electro-elastic coupled system with {\it rigid} inhomogeneities.  Moreover, the investigation carried in
\cite{cherdantsevHighContrastHomogenisation2017} requires the density of body force to be small in magnitude, while the findings reported in this paper do not impose such a restrictive assumption.

As mentioned above, a rigorous periodic homogenization was previously performed for dielectric elastomer composites in
\cite{francfortEnhancementElastodielectricsHomogenization2021,tianDielectricElastomerComposites2012}. In
\cite{tianDielectricElastomerComposites2012}, a strong assumption was made on the integrability of the electric field, while in
\cite{francfortEnhancementElastodielectricsHomogenization2021} the authors also considered the enhancement effect, which appears when adequate active space charges are introduced. However, neither
\cite{francfortEnhancementElastodielectricsHomogenization2021,tianDielectricElastomerComposites2012}
considered the high-contrast case. 

The resulting homogenized response is represented by a fully decoupled system of electrostatics and two elasticity PDEs \eqref{eq:78}, one of which includes an electrostriction term. There are two elasticity equations in the homogenized response, which is due to the resulting homogenized elasticity displacement being decomposed into two components: one solving a problem with uniformly bounded coefficients and the other corresponding to the {\it high-contrast} coefficients of the original problem.

This paper is organized as follows. In Section~\ref{sec:formulation}, the main definitions and notation are introduced and the formulation of the fine-scale problem is discussed. Auxiliary facts are reviewed in Section~\ref{sec:some-results-two}. Our main result is stated in Section~\ref{sec:main-result}, and the conclusions are given in Section~\ref{sec:conclusions}.

\section{Formulation}
\label{sec:formulation}
\subsection{Notation}
Throughout this paper, scalar-valued functions, such as $h$ in \eqref{eq:p423}, are written in usual fonts, while vector-valued or tensor-valued functions, such as the displacement $\uu$ and the elastic tensor $\bfR$ in \eqref{eq:5}, are written in bold.  Sequences are indexed by superscripts ($\phi^i$), while elements of vectors or tensors are indexed by numeric subscripts ($x_i$).  We denote by $\mathds{1}_{A}$ the characteristic function of the set $A$. Finally, the Einstein
summation convention is used whenever applicable; $\delta_{ij}$ is the Kronecker delta, and $\epsilon_{ijk}$ is the permutation symbol.

\subsection{Set up of the problem}
\label{ss:setup}
Consider $\Omega \subset \RR^d$, for $d \ge 2$, a simply connected and bounded domain of class $C^{1,1}$, and let
$Y\coloneqq [0,1]^d$ be the unit cell in $\RR^d$. The unit cell $Y$ is decomposed into:
$$Y=Y_s\cup Y_f \cup \Gamma,$$
where $Y_s \subset\subset (0,1)^d$, representing the domain occupied by the inclusion, and $Y_f$,
representing the matrix material, are open sets in
$\mathbb{R}^d$, and $\Gamma$ is the closed $C^{1,1}$ interface that
separates them. Let $i = (i_1, \ldots, i_d) \in \ZZ^d$ be a vector of indices and $\{\ee^1, \ldots, \ee^d\}$ be the canonical basis of $\RR^d$. 
For a fixed small $\varepsilon > 0,$ we define the dilated sets: 
\begin{align*}
    Y^\varepsilon_i 
    \coloneqq \varepsilon (Y + i),~~
    Y^\varepsilon_{i,s}
    \coloneqq \varepsilon (Y_s + i),~~
    Y^\varepsilon_{i,f}
    \coloneqq \varepsilon (Y_f + i),~~
    \Gamma^\varepsilon_i 
    \coloneqq \partial Y^\varepsilon_{i,s}.
\end{align*}
Typically, in homogenization theory, the positive number $\varepsilon \ll 1$ is referred to as the {\it size of the microstructure}. The effective or homogenized response of the given suspension
corresponds to the case $\varepsilon=0$, whose derivation and justification is the main focus of this paper. 

We denote by $\nn_i,~\nn_{\Gamma}$ and $\nn_{\partial \Omega}$ the unit normal vectors to $\Gamma^{\varepsilon}_{i}$ pointing outward $Y^\varepsilon_{i,s}$, on $\Gamma$ pointing outward $Y_{s}$ and on $\partial \Omega$ pointing outward, respectively; and also, we denote by $\di \hmeas$ the $(d-1)$-dimensional Hausdorff measure.
In addition, we define the sets:
\begin{align}
  \label{eq:116}
    I^{\varepsilon} 
    \coloneqq \{ 
    i \in \ZZ^d \colon Y^\varepsilon_i \subset \Omega
    \},~~
    \Omega_s^{\varepsilon} 
    \coloneqq \bigcup_{i\in I^\varepsilon}
Y_{i,s}^{\varepsilon},~~
    \Omega_f^{\varepsilon} 
    \coloneqq \Omega \setminus \Omega_s^{\varepsilon},~~
    \Gamma^\varepsilon 
    \coloneqq \bigcup_{i \in I^\varepsilon} \Gamma^\varepsilon_i.
\end{align}
see \cref{fig:1}.

\begin{figure}[ht]
\centering
% \def\svgwidth{0.5\columnwidth}
% \import{figures/}{janus_1.pdf_tex}
\includegraphics[width=0.5\textwidth]{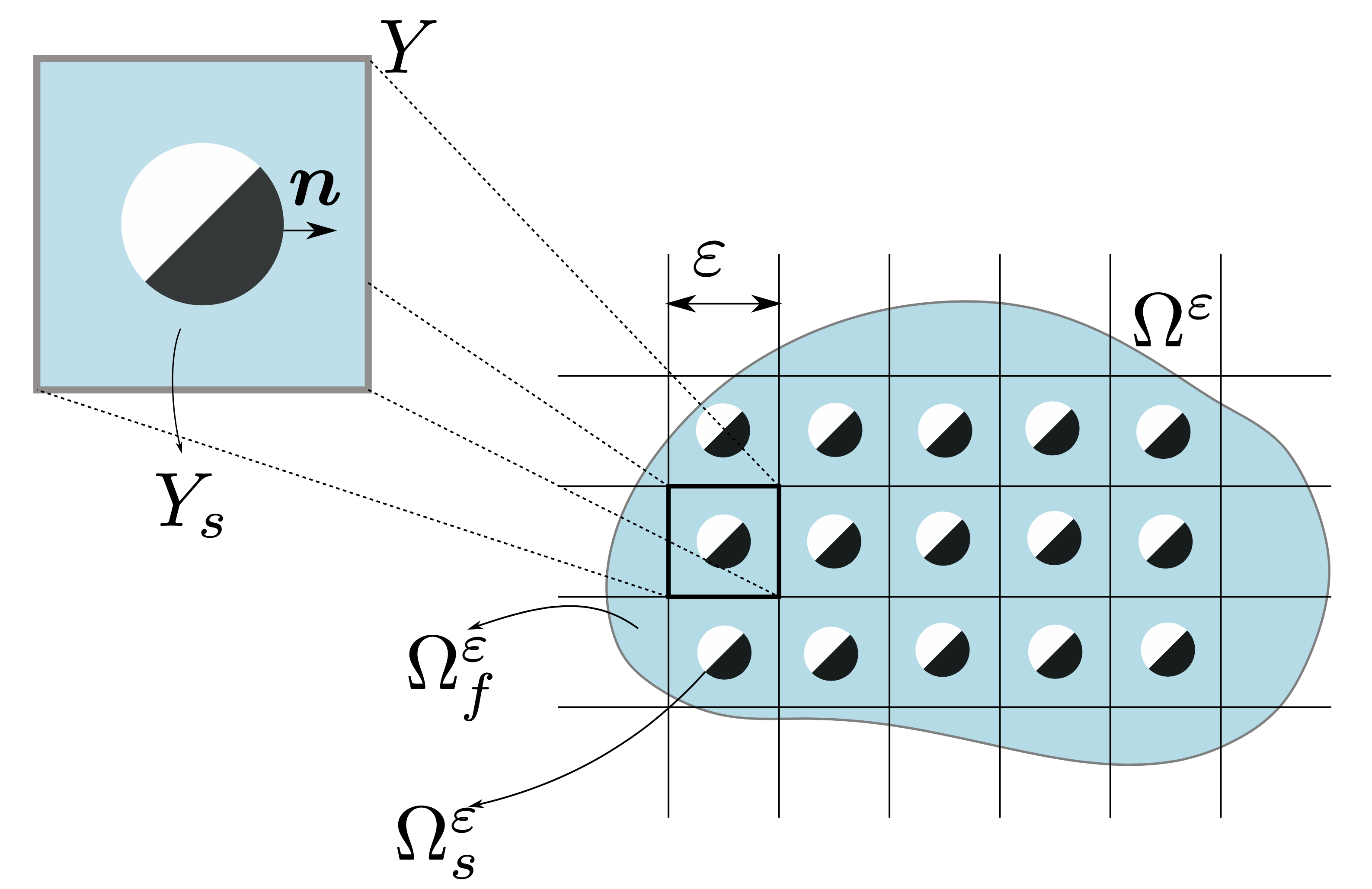}
\caption{Reference cell $Y$ and domain $\Omega$.}
\label{fig:1}
\end{figure}

Given a matrix $\bfr$ of size $d\times d$, we introduce the following definitions:
\begin{enumerate}[label=(A{\arabic*}),ref=\textnormal{(A{\arabic*})}]
\item \label{cond:a-periodic} $Y-$periodicity: for all $z
  \in \RR^d$, for all $m \in \ZZ$, and for all $k \in \{ 1,\ldots, d \}$ we have 
\begin{equation*}
%\label{eq:3}
\bfr (z + m \ee^k) = \bfr (z).
\end{equation*}
  \item\label{cond:a-bound} Boundedness: there exists $\Lambda > 0$ such that
\begin{equation*}
%\label{eq:2}
\norm{\bfr}_{L^{\infty}} \le \Lambda.
\end{equation*}
  \item\label{cond:a-elliptic} Ellipticity: there exists
    $\lambda > 0$ such that for all $\xi \in \RR^d$, for all $x \in
    \RR^d$, we have
\begin{equation*}
%\label{eq:1}
\bfr (x) \xi \cdot \xi \ge \lambda \abs{\xi}^2.
\end{equation*}
\item \label{cond:a-holder-cont} Piecewise H\"{o}lder continuity: there exists $\alpha \in (0,1)$ and a
  partition $\left\{ D_i \right\}_{i=1}^n$ of $Y$ such that $\bfr$ is
  $C^{\alpha}-$H\"{o}lder continuous on each $D_i$, $1 \le i \le n$.\\
\end{enumerate}

Denote by $\fraM (\lambda,\Lambda)$ the set of matrices that satisfy
\ref{cond:a-bound}-\ref{cond:a-elliptic}, and by
$\fraM_{\sym,\per}(\lambda,\Lambda)$ the subset of symmetric matrices in $\fraM(\lambda,\Lambda)$
that also satisfy \ref{cond:a-periodic}. 

We also introduce a few more definitions necessary for the elasticity equation. For $\uu \in H^1(\Omega,\RR^d)$, we define the symmetric gradient (also called the linearized strain tensor) by
\begin{align}
\label{eq:6}
\DD(\uu) \coloneqq \frac{\nabla \uu + \nabla \uu^{\top}}{2}.
\end{align}
Here, we recall that 
for a fourth-order tensor $\bfQ = \left( \bfQ_{ijkh} \right)_{1 \le
  i,j,k,h \le d}$ and two matrices $\bfc$ and $\bfd$, we define
$$\bfQ : \bfc \coloneqq \left( (\bfQ_{ijkh} \bfc_{kh})_{ij} \right)_{1
  \le i,j \le d}$$ and
$$\bfQ : \bfc : \bfd \coloneqq \bfQ_{ijkh} \bfc_{ij} \bfd_{kh},$$ for $1\le i,j,k,h\le d$. We say
that a fourth-order tensor $\bfQ$ is symmetric if it is both
\begin{itemize}
\item major symmetric, i.e., 
\begin{align*}
\bfQ_{ijkh} = \bfQ_{khij}, \text{ and }
\end{align*}
\item minor symmetric, i.e., 
  \begin{align*}
    \bfQ_{ijkh} = \bfQ_{ijhk} \text{ and } \bfQ_{ijkh} = \bfQ_{jikl},
\end{align*}
\end{itemize}
% $$\bfQ_{ijkh} = \bfQ_{jikh}=\bfQ_{ijhk}$$ 
for any $1 \le i,j,k,h \le
d$.

Let $0 < \lambda_e < \Lambda_e < \infty$. We denote by
$\fraM^e(\lambda_e, \Lambda_e)$ the set of all fourth-order tensors
$\bfQ = \left( \bfQ_{ijkh} \right)_{1 \le i,j,k,h \le d}$ satisfying
$\norm{\bfQ}_{L^{\infty}} \le \Lambda_e$ and the following strong ellipticity condition (which is a generalization of
\ref{cond:a-elliptic} to fourth-order tensors)
\begin{align}
\label{eq:36}
\tag{A3'} \bfQ:\bfc:\bfc \ge \lambda_e \abs{\bfc}^2, \qquad \text{ for
  any matrix } \bfc \in \RR^{d \times d}.
\end{align}
We denote by $\fraM^e_{\per}(\lambda_e,\Lambda_e)$ (resp.
$ \fraM^e_{\sym}(\lambda_e,\Lambda_e)$) the subset of
$\fraM^e(\lambda_e,\Lambda_e)$ consisting of only $Y-$periodic tensors
(resp. symmetric tensors) and we define
$$\fraM^e_{\sym,\per} \coloneqq \fraM^e_{\sym} \cap \fraM^e_{\per}.$$

In the high-contrast elastic problem, we let
 $\gamma > 0,~\bfB, \bfR \in \fraM^e_{\sym,\per}(\lambda_e, \Lambda_e)$ and we define the elastic tensor by
\begin{align}
\label{eq:5}
\bfB^{\varepsilon} \left( y \right) \coloneqq \bfB(y)  + 
   \frac{1}{\varepsilon^{2\gamma}}  \bfR (y) \mathds{1}_{Y_s}(y), \quad\text{ for all } y \in Y.
\end{align}

\subsection{The fine-scale coupled system}
\label{sec:fine-scale-coupled}
Denote by $\bfC \in \fraM^e_{\sym,\per}(\lambda_e, \Lambda_e)$ a fourth-order tensor that describes the inherent electrostriction of the material, i.e., without taking into account the change in the
electrostatic field resulting from the deformation. Let $\bg \in L^2(\Omega,\RR^d)$, $f \in L^{\infty}(\Omega)$, $h \in C^{1,1}(\partial\Omega)$,  and
$\bfa \in \fraM_{\sym,\per}(\lambda,\Lambda)$ also satisfying \ref{cond:a-holder-cont}. The fine-scale displacement $\uu^{\varepsilon}$ and the electrostatic potential
$\varphi^{\varepsilon}$ satisfy the following coupled system \cite{tianDielectricElastomerComposites2012,francfortEnhancementElastodielectricsHomogenization2021}:
\begin{subequations}
  \label{eq:p410} 
\begin{align}
\label{eq:p411}
  -\Div \left[ \bfB^{\varepsilon} \left( \frac{x}{\varepsilon} \right): 
  \DD \left( 
  \uu^{\varepsilon}  \right) + 
  \bfC \left( \frac{x}{\varepsilon} \right): \left(\nabla \varphi^{\varepsilon} \otimes \nabla \varphi^{\varepsilon} \right) \right]
  &= \bg && \text{ in } \Omega,\\
  \label{eq:p421}
  -\Div \left[ \bfa \left( \frac{x}{\varepsilon} \right) \nabla
  \varphi^{\varepsilon} \right]
  &=  f && \text{ in }\Omega,
\end{align}
\end{subequations}
together with the boundary conditions:
\begin{subequations}
  \label{eq:p415} 
  \begin{align}
     \label{eq:p423}
  \varphi^{\varepsilon}
  &= h, \text{ on }\partial \Omega,\\
\label{eq:p416}
  \uu^{\varepsilon}
  &= 0, \text{ on }\partial \Omega.
\end{align}
\end{subequations}

The system \eqref{eq:p410}-\eqref{eq:p415} implies the following
balance equations:
\begin{subequations}
 \label{eq:p413}
\begin{align}
  \label{eq:p414}
  \int_{\Gamma_{i}^{\varepsilon}} \left\llbracket
  \bfB^{\varepsilon} \left( \frac{x}{\varepsilon} \right):  \DD (\uu^{\varepsilon}) + 
  \bfC \left( \frac{x}{\varepsilon} \right) : \left( \nabla \varphi^{\varepsilon} \otimes \nabla \varphi^{\varepsilon} \right)\right\rrbracket \nn_i\di \hmeas
  &=0,\\
  \label{eq:p422}
  \int_{\Gamma_{i}^{\varepsilon}} \left\{ \left\llbracket
  \bfB^{\varepsilon} \left( \frac{x}{\varepsilon} \right):  \DD
  (\uu^{\varepsilon}) + 
   \bfC \left( \frac{x}{\varepsilon} \right):  \left( \nabla \varphi^{\varepsilon} \otimes \nabla \varphi^{\varepsilon} \right)\right\rrbracket\nn_i \right\} \times \nn_i\di
  \hmeas
  &=0,
\end{align}
\end{subequations}
where $\left\llbracket \, \cdot \, \right\rrbracket$ denotes the jump
on the interfaces $\Gamma^{\varepsilon}_i$.

\section{Two-scale convergence}
\label{sec:some-results-two}
Two-scale convergence was conceived by G. Nguetseng \cite{nguetsengGeneralConvergenceResult1989} and
further developed by G. Allaire \cite{allaireHomogenizationTwoscaleConvergence1992}.  Here, we collect important notions
and results relevant to this paper, whose proofs can be found in \cite{allaireHomogenizationTwoscaleConvergence1992,nguetsengGeneralConvergenceResult1989}. The following spaces are used later in this paper.
\begin{itemize}[wide]
    \item $C_{\per}(Y)$ -- the subspace of $C(\RR^d)$ of $Y$-periodic functions;
    \item $C^{\infty}_{\per}(Y)$ -- the subspace of $C^{\infty}(\RR^d)$ of $Y$-periodic functions;
    \item $H^1_{\per}(Y)$ -- the closure of $C^{\infty}_{\per}(Y)$ in the $H^1$-norm;
    \item
    
    $\mathcal{D}(\Omega, X)$ -- where $X$ is a Banach space -- the space  infinitely differentiable functions from $\Omega$ to $X$, whose  support is a compact set of $\mathbb{R}^d$ contained in $\Omega$.

    \item $L^p(\Omega, X)$ -- where $X$ is a Banach space and $1 \le p \le \infty$ -- the space of measurable functions $w \colon x \in \Omega \mapsto w(x) \in X$ such that
    $
    \norm{w}_{L^p(\Omega, X)}
    \coloneqq \left(\int_{\Omega} \norm{w(x)}^p_{X} \di x\right)^\frac{1}{p} < \infty.
    $
     
    \item $L^p_{\per}\left(Y, C(\bar{\Omega})\right)$ -- the space of measurable functions $w \colon y \in Y \mapsto w(\cdot,y) \in C(\bar{\Omega})$, such that %$\norm{w(x,\cdot)}_{C(\bar{\Omega})} \in L^2_{\per}(Y)$.
    $w$ is periodic with respect to $y$ and
    $
    %\norm{w}_{L^2_{\per}\left(Y, C(\bar{\Omega})\right)} \coloneqq 
    %\left[
    \int_{Y} \left(\sup_{x \in \bar{\Omega}} \abs{w(x,y)}\right)^p \di y 
    %\right]^\frac{1}{2} 
    < \infty.
    $
  \end{itemize}

\begin{definition}[$L^p-$admissible test function]
\label{sec:two-scale-conv-1}
Let $1 \le p < + \infty$. A function $\psi \in
L^p(\Omega \times Y)$,  $Y$-periodic in the second component, is called an $L^p-$admissible test function if for
all $\varepsilon > 0$,
$\psi \left( \cdot, \frac{\cdot}{\varepsilon} \right)$ is measurable and
\begin{align}
\label{eq:23}
\lim_{\varepsilon \to 0} \int_{\Omega} \abs{\psi \left( x,
  \frac{x}{\varepsilon} \right)}^p \di x = \frac{1}{\abs{Y}}
  \int_{\Omega} \int_Y \abs{\psi (x,y)}^p \di y \di x.
\end{align}
\end{definition}

It is known that functions belonging to the spaces $\mathcal{D} \left(\Omega,
  C_\per^\infty (Y)\right)$, $C \left( \bar{\Omega}, C_{\per}(Y)
\right)$, $L^p_{\per}\left( Y,  C(\bar{\Omega})\right)$ or $L^p\left(\Omega, C_{\per}(Y)\right)$ are admissible \cite{allaireHomogenizationTwoscaleConvergence1992}, but the precise
characterization of those admissible test functions is still an
open question.
\begin{definition}
  A sequence $\{ v^\varepsilon \}_{\varepsilon>0}$ in $L^2(\Omega)$ is
  said to \emph{two-scale converge} to $v = v(x,y)$, with
  $v \in L^2 (\Omega \times Y)$, and we write
  $v^\varepsilon \tscale v$, if and only if:
  $\{ v^\varepsilon \}_{\varepsilon>0}$ is bounded and
\begin{align}
\label{eq:2sc}
    \lim_{\varepsilon \to 0} \int_\Omega v^\varepsilon(x) \psi \left( x, \frac{x}{\varepsilon}\right) \di x 
    = \frac{1}{\abs{Y}} \int_\Omega \int_Y v(x,y) \psi(x,y) \di y \di x,
\end{align}
for any test function $\psi = \psi (x, y)$, with $\psi \in
\mathcal{D} \left(\Omega, C_\per^\infty (Y)\right)$.
\end{definition}
In \eqref{eq:2sc}, we can choose $\psi$ to be any ($L^2-$)admissible test
function. Any bounded sequence $v^{\varepsilon}\in L^2(\Omega)$ has
a subsequence that two-scale converges to a limit $v^0 \in L^2(\Omega
\times Y)$. 

To obtain the limit of the coupled equation, the following corrector result for the electrostatic potential $\varphi^{\varepsilon}$ is needed: % that will be used below.
\begin{theorem}[Corrector result]\cite[Theorem 1.8, Remark 1.10 and
Corollary 5.4]{allaireHomogenizationTwoscaleConvergence1992}
\label{sec:two-scale-corrector}
Let $u^{\varepsilon}$ be a sequence of functions in $L^2(\Omega)$ that
two-scale converges to a limit $u^0 (x,y) \in L^2(\Omega\times
Y)$. Assume that 
\begin{align}
\label{eq:24}
\lim_{\varepsilon \to 0} \norm{u^{\varepsilon}}_{L^2(\Omega)} =
  \norm{ 
  u^0
  }_{L^2(\Omega\times Y)}.
\end{align}
Then for any sequence $v^{\varepsilon}$ in $L^2(\Omega)$ that
two-scale converges to $v^0 \in L^2(\Omega \times Y),$ one has 
\begin{align}
\label{eq:25}
u^{\varepsilon} v^{\varepsilon} \wcv \frac{1}{\abs{Y}} \int_Y u^0(x,y)
  v^0(x,y) \di x \di y \text{ in } \calD '(\Omega).
\end{align}
Furthermore, if $u^0(x,y)$ % is admissible,
belongs to $L^2 \left( \Omega, C_{\per}(Y) \right)$ or $L^2_{\per} \left( Y, C(\Omega) \right)$,
then 
\begin{align}
\label{eq:29}
\lim_{\varepsilon \to 0}  \norm{u^{\varepsilon}(x) - u^0 \left( x,
  \frac{x}{\varepsilon} \right)}_{L^2(\Omega)} = 0.
\end{align}
\end{theorem}
In fact, the smoothness assumption on $u^0$ in \eqref{eq:29}  is needed
only for $u^0 \left( x, \frac{x}{\varepsilon}
\right)$ to be measurable and to belong to $L^2(\Omega)$. 

\section{Main results}
\label{sec:main-result}
\begin{lemma}[A priori estimates, existence and uniqueness]
  \label{sec:main-results-1}
  For each $\varepsilon > 0$, the system \eqref{eq:p410}-\eqref{eq:p415} has a unique weak
  solution $(\varphi^{\varepsilon}, \uu^{\varepsilon}) \in H^1(\Omega)
  \times H_0^1(\Omega,\RR^d)$ that satisfies
\begin{align}
\label{eq:27}
  \begin{split}
    &\norm{\uu^{\varepsilon}}_{H_0^1 (\Omega,\RR^d)} +  \norm{\varphi^{\varepsilon}}_{H^1(\Omega)}\\
  &\quad\le C \left(1+
    \norm{\bg}_{L^2(\Omega,\RR^d)} + \norm{h}_{C^{1,1}(\partial
    \Omega)}^2 + \norm{f}_{L^{\infty}(\Omega,\RR^d)}^2 \right),
  \end{split}
\end{align}
where $C$ is independent of $\varepsilon$.
\end{lemma}

\begin{proof}[Proof of \cref{sec:main-results-1}]
  \begin{enumerate}[wide]
  \item Let $\Phi \in H^1(\Omega)$ such that $\mathrm{trace}\, \Phi = h$ on $\partial \Omega$. We look
for $\varphi^{\varepsilon} \in H^1(\Omega)$ and
$\uu^{\varepsilon} \in H_0^1(\Omega,\RR^d)$  such that for all
$\eta \in H_0^1(\Omega)$ and $\bm{\xi} \in C_c^{\infty}(\Omega, \RR^d)$,
the following holds
\begin{align}
  \label{eq:r5}
  \begin{split}
    &\int_{\Omega} \bfa \left( \frac{x}{\varepsilon} \right) \nabla
    \left( \varphi^{\varepsilon} - \Phi \right) \cdot \nabla \eta \di x 
    + \int_{\Omega} \bfB^{\varepsilon} \left( \frac{x}{\varepsilon} \right): \DD(\uu^{\varepsilon}) : \DD(\bm{\xi})
    \di x 
   \\ &\qquad+
    \int_{\Omega} \bfC \left( \frac{x}{\varepsilon} \right) : \left(  \nabla \varphi^{\varepsilon} \otimes \nabla \varphi^{\varepsilon} \right) : \DD(\bm{\xi})\di x
   \\ &\qquad=  \int_{\Omega} \bg \cdot \bm{\xi} \di x 
    + \int_{\Omega} f \eta \di x
    -\int_{\Omega}
    \bfa \left( \frac{x}{\varepsilon} \right)
    \nabla \Phi \cdot \nabla \eta \di x.
\end{split}
\end{align}

\item By letting $\bm{\xi} = 0$, we look for a unique solution
$\varphi^{\varepsilon} \in H^1(\Omega)$ of the electrostatic problem

\begin{align}
  \label{eq:r6}
  \int_{\Omega} \bfa \left( \frac{x}{\varepsilon} \right)
  \nabla \left( \varphi^{\varepsilon} - \Phi \right)\cdot \nabla \eta \di x
  =\int_{\Omega} f \eta \di x
  -\int_{\Omega}
   \bfa \left( \frac{x}{\varepsilon} \right)
    \nabla \Phi  \cdot\nabla \eta \di x, \text{ for all }
  \eta \in H_0^1(\Omega).
\end{align}

From the Lax-Milgram theorem and the trace theorem, \eqref{eq:r6} has a
unique solution $\varphi^{\varepsilon} \in H^1(\Omega)$ such that 
\begin{align*}
% \label{eq:26o}
  \norm{\varphi^{\varepsilon}}_{H^1(\Omega)}
  \le C \left( \norm{h}_{C^{1,1}(\partial \Omega)} + \norm{f}_{L^{\infty}(\Omega)} \right)
\end{align*}
for some $C > 0$, independent of $\varepsilon$. Moreover, by \cite[Theorem
1]{dangGlobalGradientEstimate2022} (see also \cite{francfortEnhancementElastodielectricsHomogenization2021}), $\varphi^{\varepsilon}$ 
actually belongs to $W^{1,\infty}(\Omega)$ and
\begin{align}
\label{eq:26}
\norm{\nabla \varphi^{\varepsilon}}_{L^{\infty}(\Omega)}
  \le C \left( \norm{h}_{C^{1,1}(\partial \Omega)} + \norm{f}_{L^{\infty}(\Omega)} \right).
\end{align}
Therefore, we have $ \nabla
\varphi^{\varepsilon} \otimes \nabla \varphi^{\varepsilon} \in
L^{\infty}(\Omega,\RR^{d \times d})$.

\item By letting $\eta = 0$, we look for $\uu^{\varepsilon} \in H_0^1(\Omega,\RR^d)$ such
that for any $\bm{\xi} \in C_c^{\infty}(\Omega,\RR^d)$, we have
\begin{align}
\label{eq:r7}
\int_{\Omega} \bfB^{\varepsilon} \left( \frac{x}{\varepsilon} \right): \DD(\uu^{\varepsilon}):\DD(\bm{\xi})
  \di x 
  =\int_{\Omega} \bg \cdot \bm{\xi} \di x
  - \int_{\Omega} \bfC \left( \frac{x}{\varepsilon} \right): \left(  \nabla \varphi^{\varepsilon} \otimes \nabla \varphi^{\varepsilon} \right) : \DD(\bm{\xi})\di x.
\end{align}
For $\uu,\bm{\xi}$ in $H^1_0(\Omega,\RR^d)$, we define 
\begin{align*}
  \calB^{\varepsilon} (\uu, \bm{\xi})
  &\coloneqq \int_{\Omega} \bfB^{\varepsilon} \left( \frac{x}{\varepsilon} \right): \DD(\uu^{\varepsilon}) : \DD(\bm{\xi})
    \di x,\\
  \calL^{\varepsilon}(\bm{\xi})
  &\coloneqq
    \int_{\Omega} \bg \cdot \bm{\xi} \di x
  - \int_{\Omega} \bfC \left( \frac{x}{\varepsilon} \right): \left(  \nabla \varphi^{\varepsilon} \otimes \nabla \varphi^{\varepsilon} \right): \DD(\bm{\xi})\di x.
\end{align*}
Then, for each $\varepsilon > 0$, $\calB^{\varepsilon}$ is a continuous
bilinear form on $H^1_0(\Omega,
\RR^d) \times H_0^1(\Omega,\RR^d)$ and $\calL^{\varepsilon}$ is a
linear form on $H_0^1(\Omega,\RR^d)$. Indeed, by H\"{o}lder's
inequality and \eqref{eq:26}, there exists a
constant $C > 0$ independent of $\varepsilon$ such that
\begin{align}
  \label{eq:17}
  \begin{split}
    \abs{\calB^{\varepsilon}(\uu,\bm{\xi})}
  &\le \norm{\bfB^{\varepsilon}}_{L^{\infty}(\Omega)} \norm{\DD(\uu)}_{L^2(\Omega)}
    \norm{\DD(\bm{\xi})}_{L^2(\Omega)}\\
    &\quad\le C \left( 1 + \frac{1}{\varepsilon^{2\gamma}} \right) \Lambda_e
    \norm{\uu}_{H_0^1(\Omega)} \norm{\bm{\xi}}_{H_0^1(\Omega)},\\
  \abs{\calL^{\varepsilon}(\bm{\xi})}
  &\le \norm{\bg}_{L^2(\Omega)} \norm{\bm{\xi}}_{L^2(\Omega)} + \Lambda_e\norm{\nabla
    \varphi^{\varepsilon}}_{L^{\infty}(\Omega)}^2 \norm{\DD(\bm{\xi})}_{L^{2}(\Omega)}\\
  &\quad\le C \norm{\bm{\xi}}_{H_0^1(\Omega)} \left( \norm{\bg}_{L^2(\Omega,\RR^d)} +
    \norm{h}^2_{C^{1,1}(\partial \Omega)} + \norm{f}_{L^{\infty}(\Omega)}^2 \right).
  \end{split}
\end{align}

\item We claim that $\calB^{\varepsilon}$ is coercive. Indeed, by
  \eqref{eq:5} and  Korn's inequality, we have
\begin{align}
  \label{eq:16}
  \begin{split}
    \calB^{\varepsilon} (\uu, \uu)
  &= \int_{\Omega} \bfB^{\varepsilon} \left( \frac{x}{\varepsilon}
    \right) : \DD(\uu):\DD(\uu) \di x\\
  &= \int_{\Omega} \left( \bfB\left( \frac{x}{\varepsilon}
    \right) + 
 \frac{1}{\varepsilon^{2\gamma}} \bfR \left( \frac{x}{\varepsilon}
    \right)
    \mathds{1}_{Y_s}\left( \frac{x}{\varepsilon}
    \right) \right):  \DD (\uu) :\DD(\uu)\di x\\
  &= \int_{\Omega}  \bfB\left( \frac{x}{\varepsilon}
    \right):  \DD (\uu) :\DD(\uu)\di x +
    \frac{1}{\varepsilon^{2\gamma}} \int_{\Omega_s^{\varepsilon}} \bfR
    \left( \frac{x}{\varepsilon}
    \right) : \DD(\uu):\DD(\uu)\di x\\
  &\ge \lambda_e \norm{\DD(\uu)}_{L^2(\Omega,\RR^{d\times d})}^2 + \lambda_e
    \norm{\frac{1}{\varepsilon^{\gamma}}
    \mathds{1}_{\Omega_s^{\varepsilon}}
    \DD(\uu)}^2_{L^2(\Omega,\RR^{d\times d})}\\
    &\ge C \norm{\uu}_{H_0^1 (\Omega,\RR^d)}^2,
  \end{split}
\end{align}
where $C$ is independent of $\varepsilon$. We conclude from the
Lax-Milgram theorem that, for each $\varepsilon > 0$, \eqref{eq:r7}
has a unique solution $\uu^{\varepsilon} \in H_0^1(\Omega,\RR^d).$ 
\item We now prove a uniform estimate for $\uu^{\varepsilon}$ with
  respect to $\varepsilon$. By density, we can choose $\bm{\xi} =
  \uu^{\varepsilon}$ in \eqref{eq:r7} to obtain 
\begin{align*}
\calB^{\varepsilon} (\uu^{\varepsilon}, \uu^{\varepsilon}) = \calL^{\varepsilon}(\uu^{\varepsilon}).
\end{align*}
Now by \eqref{eq:17} and \eqref{eq:16}, we have
\begin{align*}
  \norm{\uu^{\varepsilon}}^2_{H_0^1 (\Omega,\RR^d)}
  &\le C \norm{\uu^{\varepsilon}}_{H_0^1 (\Omega,\RR^d)} \left(
    \norm{\bg}_{L^2(\Omega,\RR^d)} + \norm{h}_{C^{1,1}(\partial
    \Omega)}^2 + \norm{f}_{L^{\infty}(\Omega,\RR^d)}^2 \right),
\end{align*}
from where we obtain
\begin{align}
\label{eq:18}
   \norm{\uu^{\varepsilon}}_{H_0^1 (\Omega,\RR^d)}
  &\le C  \left(
    \norm{\bg}_{L^2(\Omega,\RR^d)} + \norm{h}_{C^{1,1}(\partial
    \Omega)}^2 + \norm{f}_{L^{\infty}(\Omega,\RR^d)}^2 \right).
\end{align}
  \end{enumerate}
\end{proof}

\begin{theorem}[Homogenization theorem]
\label{thm:main-results}
For every $\varepsilon > 0$, the system \eqref{eq:p410}-\eqref{eq:p415} has a unique
solution
$(\varphi^{\varepsilon}, \uu^{\varepsilon}) \in H^1(\Omega) \times
H^1_0(\Omega,\RR^d)$.  Moreover, there exist a constant, symmetric and
elliptic matrix $\bfa^{\hom}$; three constant, symmetric and elliptic
fourth rank tensors $\bfB^{\hom}$, $\bfR^{\hom}$, $\bfT^{\hom}$; and a constant symmetric
fourth rank tensor $\bfC^{\hom}$ such that, as
  $\varepsilon\to0$, we have
\begin{align}
  \label{eq:80}
  \varphi^{\varepsilon} \wcv \varphi^0 \text{ in } H^1(\Omega),
  \quad
  \uu^{\varepsilon}
  \wcv \uu^0 \coloneqq \vv^0 + \ww^0 \text{ in } H_0^1(\Omega,\RR^d),
\end{align}
where $\varphi^0 \in H^1(\Omega)$ and $\vv^0, \ww^0 \in H_0^1(\Omega,\RR^d)$ satisfy the following effective system of equations, all defined on the domain $\Omega$,
\begin{align}
\label{eq:78}
\begin{split}
  -\Div \left( \bfa^{\hom} \nabla \varphi^0 \right)
  &= f \text{ in }\Omega, \quad \varphi^0 = h \text{ on } \partial
    \Omega,\\
 -\Div \left( \bfB^{\hom} : \DD (\vv^0)  \right)
  &= \bg  + \Div \left( \bfC^{\hom} : \left( \nabla
  \varphi^0 \otimes \nabla \varphi^0\right) \right)\text{ in }\Omega, \quad \vv^0 = 0 \text{ on } \partial
    \Omega,\\
   -\Div \left( \bfR^{\hom}: \DD(\ww^0)  \right)
  &=  \Div \left( \bfT^{\hom}: \DD(\vv^0) 
   \right)  \text{ in }\Omega, \quad \ww^0 = 0 \text{ on } \partial
    \Omega.
 \end{split}
\end{align}
\end{theorem}
%Let 
%\begin{align*}
%  \kk
%  &\coloneqq \bg + \Div \left( \bfC^{\hom} \left( \nabla \varphi^0
%    \otimes \nabla \varphi^0 \right) \right),\\
%  \calL_B
%  &\coloneqq -\Div \left( \bfB^{\hom}: \DD(\cdot) \right),\quad
%      \calL_T
%    \coloneqq -\Div \left( \bfT^{\hom}: \DD(\cdot) \right), \quad
%      \calL_R
%  \coloneqq -\Div \left( \bfR^{\hom}: \DD(\cdot) \right).
%\end{align*}
%Then from the last two equations of \eqref{eq:78}, we have 
%\begin{align*}
%  \vv^0
%  &= \calL_B^{-1} \kk,\\
%  \ww^0
%  &= -\calL_R^{-1} \calL_T \vv^0 = -\calL_R^{-1} \calL_T \calL_B^{-1} \kk.
%\end{align*}
%Therefore, 
%\begin{align*}
%  \uu^0
%  = \vv^0 + \ww^0 = \calL_B^{-1} \kk  -\calL_R^{-1} \calL_T \calL_B^{-1} \kk,
%\end{align*}
%or equivalently, 
%\begin{align*}
%  \calL_B \uu^0
%  = \kk  - \calL_B\calL_R^{-1} \calL_T \calL_B^{-1} \kk.
%\end{align*}
%The term $- \calL_B\calL_R^{-1} \calL_T \calL_B^{-1} \kk$ appears due
%to the high-contrast? \ynote{Does this require revision???}
Let 
\begin{align*}
  \kk
  &\coloneqq \bg + \Div \left( \bfC^{\hom} \left( \nabla \varphi^0
    \otimes \nabla \varphi^0 \right) \right),\\
  \calL_B
  &\coloneqq -\Div \left( \bfB^{\hom}: \DD(\cdot) \right),\quad
      \calL_T
    \coloneqq -\Div \left( \bfT^{\hom}: \DD(\cdot) \right), \quad
      \calL_R
  \coloneqq -\Div \left( \bfR^{\hom}: \DD(\cdot) \right).
\end{align*}
Then from the last two equations of \eqref{eq:78}, we have 
\begin{align*}
  \vv^0
  &= \calL_B^{-1} \kk,\\
  \ww^0
  &= -\calL_R^{-1} \calL_T \vv^0 = -\calL_R^{-1} \calL_T \calL_B^{-1} \kk.
\end{align*}
Therefore, 
\begin{align*}
  \uu^0
  = \vv^0 + \ww^0 = \calL_B^{-1} \kk  -\calL_R^{-1} \calL_T \calL_B^{-1} \kk,
\end{align*}
or equivalently, 
\begin{align*}
  \calL_B \uu^0
  = \kk  - \calL_B\calL_R^{-1} \calL_T \calL_B^{-1} \kk.
\end{align*}
The term $- \calL_B\calL_R^{-1} \calL_T \calL_B^{-1} \kk$
characterizes the effect of the high-contrast part
$\frac{1}{\varepsilon^{2\gamma}} \bfR (y) \mathds{1}_{Y_s}(y)$: the
right hand side of the homogenized equation above shows that the
stiffness is taken into account by reducing the body force, which
turns out reducing the magnitude of the homogenized displacement
$\uu^0$.

\begin{proof}[Proof of \cref{thm:main-results}]
\begin{enumerate}[wide]
\item For every $\varepsilon > 0$, let  $\vv^{\varepsilon} \in
  H_0^1(\Omega,\RR^d)$ be the solution of the following equation
\begin{align}
\label{eq:8}
  \begin{split}
    \int_{\Omega} \bfB\left( \frac{x}{\varepsilon} \right): \DD(\vv^{\varepsilon}) : \DD(\bm{\xi})
    \di x 
    &=\calL^{\varepsilon}(\bm{\xi}), \qquad \forall \bm{\xi} \in H_0^1(\Omega,\RR^d),
    % \int_{\Omega} \DD ( \vv^{\varepsilon} ) : \DD (\bm{\zeta})
    % \di x
    % &= 0 \qquad \forall \bm{\zeta} \in H^1(\Omega_s^{\varepsilon},\RR^d).
  \end{split}
\end{align}
where 
\begin{align*}
\calL^{\varepsilon}(\bm{\xi})
  &\coloneqq
    \int_{\Omega} \bg \cdot \bm{\xi} \di x
  - \int_{\Omega} \bfC \left( \frac{x}{\varepsilon} \right): \left(  \nabla \varphi^{\varepsilon} \otimes \nabla \varphi^{\varepsilon} \right): \DD(\bm{\xi})\di x.
\end{align*}
% The existence of $\uu^{\varepsilon, f}$ follows the same argument as
% in \cite[Section
% 3.2]{dangHomogenizationNondiluteSuspension2021}.
The existence of $\vv^{\varepsilon}$ follows by application of the Lax-Milgram theorem.
Moreover, there
exists $C > 0$ independent of $\varepsilon$ such that 
\begin{align}
\label{eq:19}
  \norm{\vv^{\varepsilon}}_{H^1_0(\Omega,\RR^d)}
  \le C \left( {\norm{\bg}_{L^2(\Omega,\RR^d)}} +
  \norm{h}_{C^{1,1}(\partial \Omega)}^2 +\norm{f}_{L^{\infty}(\Omega)}^2 \right).
\end{align}
Clearly, \eqref{eq:8} represents the classical case without high-contrast coefficients. We now measure the
effect of the high-contrast part by defining
\begin{align*}
  \ww^{\varepsilon}
  \coloneqq \frac{\uu^{\varepsilon} -
  \vv^{\varepsilon}}{\varepsilon^{\gamma}}.
\end{align*}
We have that $\varepsilon^{\gamma}\ww^{\varepsilon} \in H_0^1(\Omega,\RR^d)$ and from
\eqref{eq:18} and \eqref{eq:19}, we obtain
\begin{align}
\label{eq:13}
\norm{\varepsilon^{\gamma} \ww^{\varepsilon}}_{H^1_0(\Omega,\RR^d)}
  \le C \left( {\norm{\bg}_{L^2(\Omega,\RR^d)}} +
  \norm{h}_{C^{1,1}(\partial \Omega)}^2 +\norm{f}_{L^{\infty}(\Omega)}^2 \right).
\end{align}

\item\label{item:1} From \eqref{eq:27} and
  \cite{allaireHomogenizationTwoscaleConvergence1992,nguetsengGeneralConvergenceResult1989}, we conclude that
  there exist $\varphi^0 \in H^1(\Omega)$, $\varphi^1 \in
  L^2(\Omega,H^1_{\per}(Y)/\RR)$,   $\uu^0, \vv^0, \ww^0 \in H_0^1(\Omega,\RR^d)$, and
  $\uu^1, \vv^1, \ww^1 \in L^2(\Omega, H^1_{\per}(Y,\RR^d)/\RR)$  such that 
\begin{align}
\label{eq:28}
  \begin{split}
    &\varphi^{\varepsilon}
      \wcv[H_0^1]  \varphi^0, \qquad
    \nabla \varphi^{\varepsilon} \wcv[2] \nabla \varphi^0 + \nabla_y \varphi^1, \\
    &\uu^{\varepsilon}\wcv[H_0^1] \uu^0, \qquad
      \nabla \uu^{\varepsilon}
      \wcv[2] \nabla \uu^0 + \nabla_y \uu^1,\\
    &\vv^{\varepsilon}\wcv[H_0^1] \vv^0, \qquad
      \nabla \vv^{\varepsilon}
      \wcv[2] \nabla \vv^0 + \nabla_y \vv^1,\\
    &\varepsilon^{\gamma}\ww^{\varepsilon}\wcv[H_0^1] \ww^0, \qquad
      \nabla \varepsilon^{\gamma}\ww^{\varepsilon}
      \wcv[2] \nabla \ww^0 + \nabla_y \ww^1.
  \end{split}
\end{align}
\item In \eqref{eq:r6}, let $$\eta (x) = \eta^0 (x) + \eta^1 \left( x,
    \frac{x}{\varepsilon} \right)$$ with $\eta^0 \in \calD \left(
    \Omega \right)$ and $\eta^1 \in \calD \left( \Omega,
    C_{\per}^{\infty}(Y)/\RR \right)$, then 
\begin{align*}
\int_{\Omega} \bfa \left( \frac{x}{\varepsilon} \right) \nabla
  \varphi^{\varepsilon}(x) \cdot \left( \nabla \eta^0 (x) + \varepsilon
  \nabla \eta^1 \left(x, \frac{x}{\varepsilon} \right) + \nabla_y
  \eta^1 \left( x, \frac{x}{\varepsilon} \right) \right) \di x =
  \int_{\Omega} f (\eta^0 + \varepsilon \eta^1) \di x.
\end{align*}
Since $\bfa \eta^i, i \in \left\{ 0,1 \right\},$ belong to $L^p_{\per}
\left( Y, C (\bar{\Omega}) \right)$, they are admissible in the sense
of \cref{sec:two-scale-conv-1}. As $\varepsilon \to 0$, we get from \eqref{eq:28} that
\begin{align}
\label{eq:7}
\int_{\Omega \times Y} \bfa (y) (\nabla \varphi^0 + \nabla_y
  \varphi^1) \cdot (\nabla \eta^0 + \nabla_y \eta^1) \di x \di y
  = \int_{\Omega} f \eta^0 \di x,
\end{align}
for any $\eta^0 \in \calD(\Omega)$ and
$\eta^1 \in \calD (\Omega,C_{\per}^{\infty}(Y)/\RR)$. Since $\bfa$ is elliptic, the left hand side of \eqref{eq:7} is a coercive continuous bilinear form with variable
$(\varphi^0, \varphi^1) \in H^1(\Omega) \times L^2 \left( \Omega,
  H_{\per}^1(Y)/\RR \right)$. Thus, using the Lax-Milgram theorem we obtain that
\eqref{eq:7} has a unique solution
$(\varphi^0, \varphi^1) \in H^1(\Omega) \times L^2 \left( \Omega,
  H_{\per}^1(Y)/\RR \right)$ satisfying $\varphi^0 = h$ on $\partial \Omega$.

We introduce the cell problems
\begin{align}
\label{eq:9}
\chi^i \in H_{\per}^1(Y)/\RR, \qquad -\Div \left( \bfa (y) (\e^i + \nabla \chi^i)
  \right) = 0, \qquad {1 \le i \le d.}
\end{align}
The existence and uniqueness of the solution of \eqref{eq:9} follows by application of the
Lax-Milgram theorem. At this point, we define $\bfa^{\hom} \in \RR^{d \times d}$ by
\begin{align*}
{ \left( \bfa^{\hom} \right)_{ij} }
  &\coloneqq \int_Y \bfa(y) (\e^i + \nabla_y \chi^i) \cdot (\e^j + \nabla_y
  \chi^j) \di y\\
  &=  \int_Y \bfa(y) (\e^i + \nabla_y \chi^i) \cdot \e^j  \di y,
\end{align*}
where in the last identity we used the fact that 
$\int_Y \bfa (y) \left( \e^i + \nabla_y \chi^i \right) \cdot \nabla_y
\chi^j \di y = 0$, which comes from \eqref{eq:9} and periodicity.
  
Consider the problem
\begin{align}
\label{eq:30}
  \begin{split}
  \phi^0 \in H^1(\Omega), \quad -\Div \left( \bfa^{\hom} \nabla \phi^0
  \right)
  &= f \qquad \text{ in }\Omega,\\
  \phi^0
  &= h \qquad \text{ on }\partial \Omega.    
  \end{split}
\end{align}
By the ellipticity of $\bfa^{\hom}$ and the Lax-Milgram theorem, \eqref{eq:30}
has a unique solution. We now define the function 
\begin{align*}
\phi^1(x,y) \coloneqq \frac{\partial
  \phi^0}{\partial x_i}(x) \chi^i(y),
\end{align*}
where $\phi^0$ is the unique solution of \eqref{eq:30} and $\chi^i$ is the unique solution of \eqref{eq:9}.  By a direct calculation, observe that 
\begin{align}
  \nabla \phi^0(x) + \nabla_y \phi^1(x,y)
  = \frac{\partial
  \phi^0}{\partial x_i}(x) \left( \e^i + \nabla_y \chi^i(y) \right).
  \label{simplif}
\end{align}
For any $\eta^0 \in \calD(\Omega)$ and
$\eta^1 \in \calD (\Omega,C_{\per}^{\infty}(Y)/\RR)$, by \eqref{simplif} we have
\begin{align*}
  &\int_{\Omega \times Y} \bfa (y) (\nabla \phi^0 + \nabla_y
    \phi^1) \cdot (\nabla \eta^0 + \nabla_y \eta^1) \di x \di y\\
  &= \int_{\Omega \times Y} \bfa (y) \frac{\partial
  \phi^0}{\partial x_i}(x) \left( \e^i + \nabla_y \chi^i(y)
    \right)\cdot (\nabla \eta^0(x) + \nabla_y \eta^1(x,y)) \di x \di y\\
  &=\int_{\Omega \times Y} \bfa (y) \frac{\partial
  \phi^0}{\partial x_i}(x) \left( \e^i + \nabla_y \chi^i(y)
    \right)\cdot \nabla \eta^0 (x) \di x \di y\\
  &\quad+ \int_{\Omega \times Y} \bfa (y) \frac{\partial
  \phi^0}{\partial x_i}(x) \left( \e^i + \nabla_y \chi^i(y)
    \right)\cdot  \nabla_y \eta^1(x,y) \di x \di y
\end{align*}
By integration by parts the above becomes
\begin{align*}
  &\int_{\Omega } \left( \int_Y  \bfa (y) \left( \e^i + \nabla_y \chi^i(y)
    \right) \di y\right)\frac{\partial
  \phi^0}{\partial x_i}(x) \cdot \nabla \eta^0 (x) \di x\\
  &\quad- \int_{\Omega} \int_Y \Div_y \left( \bfa (y)  \left( \e^i + \nabla_y \chi^i(y)
    \right)  \right) \frac{\partial
    \phi^0}{\partial x_i}(x) \eta^1(x,y) \di y \di x\\
  &=\int_{\Omega } \left( \int_Y  \bfa (y) \left( \e^i + \nabla_y \chi^i(y)
    \right) \cdot \e^j\di y\right)\frac{\partial
  \phi^0}{\partial x_i}(x) \frac{\partial \eta^0}{\partial x_j} (x) \di x\\
  &\quad- \int_{\Omega} \int_Y \Div_y \left( \bfa (y)  \left( \e^i + \nabla_y \chi^i(y)
    \right)  \right) \frac{\partial
  \phi^0}{\partial x_i}(x) \eta^1(x,y) \di y \di x\\
  &= \int_{\Omega} f \eta^0 \di x,
\end{align*}
where, in the last equality, we use \eqref{eq:30} to compute the first
integral and \eqref{eq:9} to cancel the second one. It follows that
$(\phi^0, \phi^1)$ also satisfies \eqref{eq:7}. Therefore, by
uniqueness, $\varphi^0 = \phi^0$ and $\varphi^1 = \phi^1$.

\item In \eqref{eq:8}, choose $\bm{\xi}(x) =  \bm{\zeta}^0 (x)+ \varepsilon \bm{\zeta}^1 \left( x,
  \frac{x}{\varepsilon} \right)$ with $\bm{\zeta}^0 \in \calD(\Omega,\RR^d)$
and $\bm{\zeta}^1 \in \calD \left(
  \Omega,C_{\per}^{\infty}(Y,\RR^d)/\RR \right)$, then 
\begin{align*}
\int_{\Omega} \bfB \left( \frac{x}{\varepsilon} \right):  \DD (
  \vv^{\varepsilon} ): \left( \DD (\bm{\zeta}^0) + \varepsilon \DD
  (\bm{\zeta}^1) + \DD_y \left( \bm{\zeta}^1 \right) \right) \di x
  &= \calL^{\varepsilon} \left(\bm{\zeta}^0 + \varepsilon \bm{\zeta}^1 \left( x,
  \frac{x}{\varepsilon} \right)\right).
\end{align*}

Since $\bfa$ is piecewise  $C^{\alpha}-$H\"{o}lder continuous, by
\cite[Theorem
3.2]{francfortEnhancementElastodielectricsHomogenization2021} (see
also \cite[Theorem 1]{dangExplicitCorrectorHomogenization2023} for the
nonlinear case) we have: 
\begin{align*}
\lim_{\varepsilon\to 0} \norm{\nabla \varphi^{\varepsilon}(\cdot) -
  \nabla \varphi^0(\cdot) - \nabla_y \varphi^1 \left( \cdot,
  \frac{\cdot}{\varepsilon} \right)}_{L^2(\Omega)}
  =0.
\end{align*}
Therefore, by the Averaging Lemma \cite{bensoussanAsymptoticAnalysisPeriodic2011},
\begin{align}
\label{eq:32}
  \lim_{\varepsilon\to 0} \norm{\nabla \varphi^{\varepsilon}}_{L^2(\Omega)}
  = \norm{
  \nabla \varphi^0 + \nabla_y \varphi^1}_{L^2(\Omega\times Y)}.
\end{align}

Letting $\varepsilon \to 0$, by \cref{sec:two-scale-corrector}, \eqref{eq:32}, and
H\"{o}lder's inequality, we obtain 
\begin{align}
  \label{eq:21}
  \begin{split}
&\int_{\Omega\times Y} \bfB (y) : ( \DD(\vv^0) + \DD_y (\vv^1) ): \left( \DD (\bm{\zeta}^0)  + \DD_y \left( \bm{\zeta}^1 \right) \right) \di x \di  y\\
&\qquad= \int_{\Omega} \bg \cdot \bm{\zeta}^0\di x\\
&\qquad \quad  - \int_{\Omega\times Y} \bfC \left(y \right): \left( ( \nabla \varphi^0
    + \nabla_y \varphi^1)\otimes (\nabla \varphi^0
    + \nabla_y \varphi^1 )\right) : \left( \DD (\bm{\zeta}^0) + \DD_y (\bm{\zeta}^1)
    \right)\di x \di y.
  \end{split}
\end{align}
Because $\bfB \in \fraM^e_{\sym,\per}(\lambda_e,\Lambda_e),$ the left
hand side of \eqref{eq:21} is a coercive bilinear form on
$H_0^1 \left( \Omega, \RR^d \right) \times L^2 \left( \Omega,
  H_{\per}^1 (Y,\RR^d)/\RR \right)$. Moreover, since
$f \in L^{\infty}(\Omega)$ and $\varphi^0$ satisfies \eqref{eq:30},
whose coefficients are constant, by a Calder\'on-Zygmund
estimate \cite{fernandez-realRegularityTheoryElliptic2022a}, we have $\nabla \varphi^0 \in L^p(\Omega)$ for any $p \in
[1, \infty)$. Therefore, \eqref{eq:21} has a unique solution.

Similarly, as in Step 3, we define the following:
\begin{itemize}%[wide]
\item Cell problems:
  \begin{align}
\label{eq:11}
  \begin{split}
    &\bfV^{ij} \in H_{\per}^1(Y,\RR^d)/\RR, \qquad\bfp^{ij} \in
      H_{\per}^1(Y,\RR^d)/\RR,\\
    &-\Div_y \bfB (y):  \DD_y(y_j \e^i - \bfV^{ij})  = 0,\\
    &-\Div_y\left(\bfB(y):  \DD_y
      (\bfp^{ij}) +  \bfC(y):   (\e^i + \nabla_y \chi^i)\otimes(\e^j + \nabla_y
      \chi^j) \right) = 0,
      % \\
  %   &\int_{\Gamma} \left\llbracket
  %     \bfB (y) : \DD_y(y_j \e^i - \bfV^{ij})
  %     % + 
  % % \bfC (y) (\e^i + \nabla_y \chi^i)\cdot(\e^j + \nabla_y
  % %     \chi^j)
  %     \right\rrbracket \nn\di \hmeas
  % =0,\\
  % &\int_{\Gamma} \left\{ \left\llbracket
  %   \bfB (y):  \DD_y(y_j \e^i - \bfV^{ij})
  %   % + 
  %  % \bfC \left( y \right) (\e^i + \nabla_y \chi^i)\cdot(\e^j + \nabla_y
  %  %  \chi^j)
  %   \right\rrbracket\nn \right\} \times \nn\di
  % \hmeas = 0,\\
  %   &\int_{\Gamma} \left\llbracket
  %     \bfB(y):  \DD_y
  %     (\bfp^{ij}) +  \bfC(y):   (\e^i + \nabla_y \chi^i)\otimes(\e^j + \nabla_y
  %     \chi^j)
  %     % + 
  % % \bfC (y) (\e^i + \nabla_y \chi^i)\cdot(\e^j + \nabla_y
  % %     \chi^j)
  %     \right\rrbracket \nn\di \hmeas
  % =0,\\
  % &\int_{\Gamma} \left\{ \left\llbracket
  %   \bfB(y):  \DD_y
  %     (\bfp^{ij}) +  \bfC(y):   (\e^i + \nabla_y \chi^i)\otimes(\e^j + \nabla_y
  %     \chi^j)
  %   % + 
  %  % \bfC \left( y \right) (\e^i + \nabla_y \chi^i)\cdot(\e^j + \nabla_y
  %  %  \chi^j)
  %   \right\rrbracket\nn \right\} \times \nn\di
  % \hmeas
  % =0,
  \end{split}
\end{align}
for $1 \le i,j \le d.$    
\item Fourth-order homogenized tensors $\bfB^{\hom}$ and
  $\bfC^{\hom}$:
\begin{align}
\label{eq:31}  
  \begin{split}
     \bfB^{\hom}_{ijmn}
  &\coloneqq \int_Y \bfB (y): \DD_y(y_j\e^i
    -\bfV^{ij}):\DD_y(y_n\e^m -\bfV^{mn}) \di y,\\
    \bfC^{\hom}_{ijmn}
  &\coloneqq \int_Y  \left[ \bfB(y):   \DD_y
      (\bfp^{ij}) + \bfC (y):  (\e^i + \nabla_y \chi^i) \otimes(\e^j +
    \nabla_y \chi^j) \right]:\DD_y(y_n \e^m) \di y.
  % \bfC^{\hom}_{ijmn}
  % &\coloneqq \int_Y  \left[ \bfB(y):   \DD_y
  %     (\bfp^{ij}) + \bfC (y):  (\e^i + \nabla_y \chi^i) \otimes(\e^j + \nabla_y \chi^j) \right]:\left[ \bfB(y):   \DD_y
  %   (\bfp^{nm}) \right.\\
  %   &\qquad \qquad + \left.\bfC (y):  (\e^n + \nabla_y \chi^n) \otimes(\e^m + \nabla_y \chi^m) \right] \di y,
  \end{split}
\end{align}
\item Homogenized equation:
  \begin{align}
\label{eq:33}
-\Div \left( \bfB^{\hom}_{ijmn} \left[ \DD(\vv^0) \right]_{ij}
  \e^m \otimes \e^n \right) = \bg + \Div \left( \bfC^{\hom}_{ijmn}
  \frac{\partial \varphi^0}{\partial x_i} \frac{\partial
  \varphi^0}{\partial x_j} \e^m \otimes \e^n \right).
\end{align}
\item Ansatz on $\vv^1$:
  \begin{align}
    \label{eq:10}
  \vv^1(x,y)
  = -\left[ \DD(\vv^0) \right]_{ij} \bfV^{ij}(y) + \frac{\partial
  \varphi^0}{\partial x_i} \frac{\partial \varphi^0}{\partial x_j} \bfp^{ij}(y).
\end{align}
\end{itemize}
A few remarks are in order: 
\begin{enumerate}
\item The existence and uniqueness of the solutions of the cell problems \eqref{eq:11}
  follow by the Lax-Milgram theorem and the fact that
  $\nabla_y \chi^i \in L^{\infty} (Y)$, cf. \cite[Proof of Theorem
  3.2]{francfortEnhancementElastodielectricsHomogenization2021}.
\item The tensors $\bfB^{\hom}$ and $\bfC^{\hom}$ are
  symmetric. First, notice that $$\DD(y_j \e^i) = \frac{1}{2} \left(
    \e^i \otimes \e^j + \e^j \otimes \e^i \right) = \DD (y_i\e^j),$$
  and, since $\bfC$ is symmetric,   $$\bfC : \bfc = \bfC : \bfc^{\top}$$ for any matrix $\bfc$.  Therefore, $\bfV^{ij} = \bfV^{ji}$ and
  $\bfp^{ij} = \bfp^{ji}$, which implies $\bfB^{\hom}$ and
  $\bfC^{\hom}$ are symmetric.
\item By using the symmetry of $\bfB^{\hom}$ and $\bfC^{\hom},$ we can
  write \eqref{eq:33} in a more elegant form. Indeed, since %  since both
  % $\bfB^{\hom}$ and $\bfC^{\hom}$ are symmetric, we have
  $\bfB_{ijmn}^{\hom}=\bfB_{mnij}^{\hom}$ and
  $\bfC_{ijmn}^{\hom} = \bfC_{mnij}^{\hom}$, we can write
\begin{align*}
\bfB^{\hom}_{ijmn} \left[ \DD(\vv^0) \right]_{ij}
  \e^m \otimes \e^n
  &= \bfB^{\hom}_{ijmn} \e^m \otimes \e^n \otimes \e^i \otimes \e^j :
    \left[ \DD(\vv^0) \right]_{pq} \e^p \otimes \e^q\\
  &= \bfB^{\hom}_{ijmn} \e^i \otimes \e^j \otimes \e^m \otimes \e^n :
    \left[ \DD(\vv^0) \right]_{pq} \e^p \otimes \e^q\\
  &= \bfB^{\hom} : \DD(\vv^0), 
\end{align*}
and, similarly, 
\begin{align*}
  \bfC^{\hom}_{ijmn}
  \frac{\partial \varphi^0}{\partial x_i} \frac{\partial
  \varphi^0}{\partial x_j} \e^m \otimes \e^n
  &= \bfC^{\hom}_{ijmn} \e^m \otimes \e^n \otimes \e^i \otimes \e^j : \frac{\partial \varphi^0}{\partial x_p} \frac{\partial
  \varphi^0}{\partial x_q} \e^p \otimes \e^q\\
  &= \bfC^{\hom} : \left( \nabla \varphi^0 \otimes \nabla \varphi^0 \right).
\end{align*}
Therefore, \eqref{eq:33} is equivalent to 
\begin{align}
  \label{eq:12}
  -\Div \left( \bfB^{\hom} : \DD (\vv^0) + \bfC^{\hom} : \left( \nabla
  \varphi^0 \otimes \nabla \varphi^0\right) \right)
  = \bg.
\end{align}
\end{enumerate}

We now check that $\vv^0$ and $\vv^1$ in \eqref{eq:33} and
\eqref{eq:10}, respectively, satisfy \eqref{eq:21}. Indeed, by
substituting the above ansatz \eqref{eq:10} into \eqref{eq:21}, we
obtain
\begin{align*}
  \begin{split}
&\int_{\Omega\times Y} \bfB (y):  \left( \DD(\vv^0) + \DD_y \left( -\left[ \DD(\vv^0) \right]_{ij} \bfV^{ij}(y) + \frac{\partial
  \varphi^0}{\partial x_i} \frac{\partial \varphi^0}{\partial x_j}
  \bfp^{ij}(y) \right) \right)\\
    &\qquad \quad \qquad :\left( \DD (\bm{\zeta}^0)  + \DD_y \left( \bm{\zeta}^1 \right) \right) \di x \di  y\\
&= \int_{\Omega} \bg \cdot \bm{\zeta}^0\di x\\
&\qquad  - \int_{\Omega\times Y} \bfC \left(y \right): \left( ( \nabla \varphi^0
    + \nabla_y \varphi^1)\otimes (\nabla \varphi^0
    + \nabla_y \varphi^1 )\right) :\left( \DD (\bm{\zeta}^0)  + \DD_y \left( \bm{\zeta}^1 \right) \right)\di x \di y.
  \end{split}
\end{align*}
In the above equation, we rewrite 
\begin{align*}
  \DD (\vv^0)
  &= \frac{1}{2} \left[ \DD(\vv^0) \right]_{ij} \e^i \otimes \e^j +
    \frac{1}{2} \left[ \DD(\vv^0) \right]_{ji} \e^i \otimes \e^j \quad
    \text{(since $\DD(\vv^0) = \DD(\vv^0)^{\top}$)}\\
  &= \left[ \DD(\vv^0) \right]_{ij} \frac{1}{2} \left( \e^i \otimes
    \e^j + e^j \otimes \e^i \right)\\
  &= \left[ \DD(\vv^0) \right]_{ij} \frac{1}{2} \left(
    \frac{\partial}{\partial y_m} y_j \e^i \otimes \e^m +
    \frac{\partial}{\partial y_n}y_j \e^n \otimes \e^i \right)\\
  &= \left[ \DD(\vv^0) \right]_{ij} \DD_y \left( y_j \e^i\right),
  \\
  \nabla \varphi^0 + \nabla_y \varphi^1
  &= \frac{\partial \varphi^0}{\partial x_i} \left( \e^i + \nabla_y \chi^i \right),
\end{align*}
then we group like-terms together to obtain
\begin{align*}
  \begin{split}
    &\int_{\Omega} \left[ \DD(\vv^0) \right]_{ij} 
      \int_Y\bfB (y) : \left( \DD_y \left( y_j\e^i-\bfV^{ij}(y) \right) \right): \left( \DD (\bm{\zeta}^0)  + \DD_y \left( \bm{\zeta}^1 \right) \right)
      \di y \di  x\\
    &\qquad+ \int_{\Omega}  \frac{\partial
      \varphi^0}{\partial x_i} \frac{\partial \varphi^0}{\partial x_j}
      \int_Y\bfB (y)  : \DD_y \left(  \bfp^{ij}(y) \right) :\left( \DD
      (\bm{\zeta}^0)  + \DD_y \left( \bm{\zeta}^1 \right) \right)  \di
      y \di  x\\
    &= \int_{\Omega} \bg \cdot \bm{\zeta}^0\di x\\
    &\qquad 
      - \int_{\Omega} \frac{\partial
      \varphi^0}{\partial x_i} \frac{\partial \varphi^0}{\partial x_j}
      \int_Y \bfC \left(y \right): ( \e^i + \nabla_y \chi^i)\otimes (\e^j + \nabla_y \chi^j) : \left( \DD (\bm{\zeta}^0)  + \DD_y \left( \bm{\zeta}^1 \right) \right) \di x \di y.
  \end{split}
\end{align*}
% or
% \begin{align*}
%   \begin{split}
%     &\int_{\Omega} \bm{\zeta}(x) \left[ \DD(\vv^0) \right]_{ij} 
%       \int_Y\bfB (y) : \left( \DD_y \left( y_j\e^i-\bfV^{ij}(y) \right) \right): \DD_y \left( \bm{\xi}(y) \right)
%       \di y \di  x\\
%     &\qquad+ \int_{\Omega}  \bm{\zeta}(x) \frac{\partial
%       \varphi^0}{\partial x_i} \frac{\partial \varphi^0}{\partial x_j}
%       \left( \int_Y \left[\bfB (y) :  \DD_y \left(  \bfp^{ij}(y) \right) + \bfC
%       \left(y \right): ( \e^i + \nabla_y \chi^i)\otimes (\e^j + \nabla_y
%       \chi^j) \right] \right.\\
%     &\qquad \quad \quad \quad \left.: \DD_y \left(  \bm{\xi}(y) \right)  \di y \vphantom{\int} \right) \di  x = 
%       0.
%   \end{split}
% \end{align*}
Due to \eqref{eq:11}, integration by parts and periodicity, we have
\begin{align*}
  &\int_Y\bfB (y) : \left( \DD_y \left( y_j\e^i-\bfV^{ij}(y) \right) \right): \DD_y \left( \bm{\zeta}^{1}(x,y) \right)
    \di y
    = 0,\\
  &\int_Y \left[\bfB (y) :  \DD_y \left(  \bfp^{ij}(y) \right) + \bfC
    \left(y \right): ( \e^i + \nabla_y \chi^i)\otimes (\e^j + \nabla_y
    \chi^j) \right] : \DD_y \left(  \bm{\zeta}^{1}(x,y) \right)  \di y= 0,
\end{align*}
for any $\bm{\zeta}^1 \in \calD \left(
  \Omega,C_{\per}^{\infty}(Y,\RR^d)/\RR \right)$.
Therefore, we have
\begin{align*}
  \begin{split}
    &\int_{\Omega} \left[ \DD(\vv^0) \right]_{ij} 
      \int_Y\bfB (y) : \left( \DD_y \left( y_j\e^i-\bfV^{ij}(y) \right) \right): \DD (\bm{\zeta}^0)  
      \di y \di  x\\
    &\qquad+ \int_{\Omega}  \frac{\partial
      \varphi^0}{\partial x_i} \frac{\partial \varphi^0}{\partial x_j}
      \int_Y\bfB (y)  : \DD_y \left(  \bfp^{ij}(y) \right) : \DD
      (\bm{\zeta}^0)   \di
      y \di  x\\
    &= \int_{\Omega} \bg \cdot \bm{\zeta}^0\di x\\
    &\qquad 
      - \int_{\Omega} \frac{\partial
      \varphi^0}{\partial x_i} \frac{\partial \varphi^0}{\partial x_j}
      \int_Y \bfC \left(y \right): ( \e^i + \nabla_y \chi^i)\otimes (\e^j + \nabla_y \chi^j) : \DD (\bm{\zeta}^0) \di y \di x.
  \end{split}
\end{align*}
Substituting the identity $\DD(\bm{\zeta}^0) = \left[
  \DD(\bm{\zeta}^0) \right]_{mn} \DD_y \left( y_n\e^m \right)$ into the above equation leads to
\begin{align*}
  \begin{split}
    &\int_{\Omega} \left[ \DD(\vv^0) \right]_{ij} \left[
      \DD(\bm{\zeta}^0) \right]_{mn} 
      \int_Y\bfB (y) : \left( \DD_y \left( y_j\e^i-\bfV^{ij}(y) \right) \right):  \DD_y \left( y_n\e^m \right)  
      \di y \di  x\\
    &\qquad+ \int_{\Omega}  \frac{\partial
      \varphi^0}{\partial x_i} \frac{\partial \varphi^0}{\partial x_j} \left[\DD(\bm{\zeta}^0) \right]_{mn}
      \int_Y\bfB (y)  : \DD_y \left(  \bfp^{ij}(y) \right) : \DD_y \left( y_n\e^m \right)  \di
      y \di  x\\
    &= \int_{\Omega} \bg \cdot \bm{\zeta}^0\di x\\
    &\qquad 
      - \int_{\Omega} \frac{\partial
      \varphi^0}{\partial x_i} \frac{\partial \varphi^0}{\partial x_j} \left[\DD(\bm{\zeta}^0) \right]_{mn}
      \int_Y \bfC \left(y \right): ( \e^i + \nabla_y \chi^i)\otimes (\e^j + \nabla_y \chi^j) : \DD_y \left( y_n\e^m \right)  \di y \di x.
  \end{split}
\end{align*}
From \eqref{eq:11} and periodicity, we have
$\int_Y \bfB (y): \DD_y(y_j\e^i -\bfV^{ij}):\DD_y(\bfV^{mn}) \di y =
0.$ Therefore, the above equation can be written succinctly as
\begin{align*}
  \begin{split}
    &\int_{\Omega} \bfB^{\hom}_{ijmn} \left[ \DD(\vv^0) \right]_{ij} \left[
      \DD(\bm{\zeta}^0) \right]_{mn} 
      \di  x
      + \int_{\Omega} \bfC^{\hom}_{ijmn}  \frac{\partial
      \varphi^0}{\partial x_i} \frac{\partial \varphi^0}{\partial x_j} \left[\DD(\bm{\zeta}^0) \right]_{mn} \di  x\\
    &= \int_{\Omega} \bg \cdot \bm{\zeta}^0\di x
  \end{split}
\end{align*}
which holds since $\vv^0$ satisfies \eqref{eq:33}.

\item From \eqref{eq:r7}, we get
\begin{align*}
  \calL^{\varepsilon}(\bm{\xi})
  &=\int_{\Omega}\bfB^{\varepsilon}\left( \frac{x}{\varepsilon} \right) : \DD (\uu^{\varepsilon}) : \DD(\bm{\xi}) \di x\\
  &= \int_{\Omega} \left(  \bfB\left( \frac{x}{\varepsilon}
    \right) + 
 \frac{1}{\varepsilon^{2\gamma}} \bfR \left( \frac{x}{\varepsilon}
    \right)
    \mathds{1}_{Y_s}\left( \frac{x}{\varepsilon}
    \right) \right) : \DD (\vv^{\varepsilon} +
    \varepsilon^{\gamma} \ww^{\varepsilon}) : \DD (\bm{\xi}) \di x \\
  &= \int_{\Omega} \bfB \left( \frac{x}{\varepsilon} \right) : \DD
    (\vv^{\varepsilon}): \DD(\bm{\xi}) \di x + \varepsilon^{\gamma}
    \int_{\Omega} \bfB \left( \frac{x}{\varepsilon} \right) : \DD
    (\ww^{\varepsilon}): \DD(\bm{\xi}) \di x\\
  &\qquad + \frac{1}{\varepsilon^{2\gamma}}
    \int_{\Omega_s^{\varepsilon}} \bfR \left( \frac{x}{\varepsilon}
    \right):  \DD (\vv^{\varepsilon}) : \DD (\bm{\xi})\di x +
    \frac{1}{\varepsilon^{\gamma}} \int_{\Omega_s^{\varepsilon}} \bfR
    \left( \frac{x}{\varepsilon} \right):  \DD (\ww^{\varepsilon}) :
    \DD(\bm{\xi}) \di x.
\end{align*}
Choosing $$\bm{\xi}(x) = \varepsilon^{2\gamma} \bm{\zeta}^0 (x)+ \varepsilon^{2 \gamma + 1} \bm{\zeta}^1 \left( x,
  \frac{x}{\varepsilon} \right)$$ with $\bm{\zeta}^0 \in \calD(\Omega,\RR^d)$
and $\bm{\zeta}^1 \in \calD \left(
  \Omega,C_{\per}^{\infty}(Y,\RR^d)/\RR \right)$, we obtain 
\begin{align*}
  &\varepsilon^{2\gamma}\calL^{\varepsilon}( \bm{\zeta}^0 + \varepsilon \bm{\zeta}^1)\\
  &= \varepsilon^{2\gamma}\int_{\Omega} \bfB \left( \frac{x}{\varepsilon} \right):  \DD
    (\vv^{\varepsilon}): \left( \DD (\bm{\zeta}^0) + \varepsilon \DD
    (\bm{\zeta}^1) + \DD_y(\bm{\zeta}^1) \right) \di x\\
  &\qquad+ \varepsilon^{3\gamma}
    \int_{\Omega} \bfB \left( \frac{x}{\varepsilon} \right) : \DD
    (\ww^{\varepsilon}): \left( \DD (\bm{\zeta}^0) + \varepsilon \DD
    (\bm{\zeta}^1) + \DD_y(\bm{\zeta}^1) \right) \di x\\
  &\qquad +
    \int_{\Omega_s^{\varepsilon}} \bfR \left( \frac{x}{\varepsilon}
    \right):  \DD (\vv^{\varepsilon}) : \left( \DD (\bm{\zeta}^0) + \varepsilon \DD
    (\bm{\zeta}^1) + \DD_y(\bm{\zeta}^1) \right)\di x \\
  &\qquad +
     \int_{\Omega_s^{\varepsilon}} \bfR
    \left( \frac{x}{\varepsilon} \right):  \DD (\varepsilon^{\gamma}\ww^{\varepsilon}) :
    \left( \DD (\bm{\zeta}^0) + \varepsilon \DD
    (\bm{\zeta}^1) + \DD_y(\bm{\zeta}^1) \right) \di x\\
  &= O(\varepsilon^{2\gamma}) + O(\varepsilon^{3\gamma})\\
  &\qquad +
    \int_{\Omega}\left( \mathds{1}_{\Omega_s^{\varepsilon}} -
  \mathds{1}_{\Omega \times Y_s} \left( x,
    \frac{x}{\varepsilon} \right) \right) \bfR \left( \frac{x}{\varepsilon}
    \right):  \DD (\vv^{\varepsilon}) : \left( \DD (\bm{\zeta}^0) + \varepsilon \DD
    (\bm{\zeta}^1) + \DD_y(\bm{\zeta}^1) \right)\di x \\
  &\qquad +
    \int_{\Omega}\mathds{1}_{\Omega \times Y_s} \left( x,
    \frac{x}{\varepsilon} \right) \bfR \left( \frac{x}{\varepsilon}
    \right):  \DD (\vv^{\varepsilon}) : \left( \DD (\bm{\zeta}^0) + \varepsilon \DD
    (\bm{\zeta}^1) + \DD_y(\bm{\zeta}^1) \right)\di x \\
  &\qquad +
    \int_{\Omega}\left( \mathds{1}_{\Omega_s^{\varepsilon}} -
  \mathds{1}_{\Omega \times Y_s} \left( x,
    \frac{x}{\varepsilon} \right) \right) \bfR \left( \frac{x}{\varepsilon}
    \right):  \DD (\ww^{\varepsilon}) : \left( \DD (\bm{\zeta}^0) + \varepsilon \DD
    (\bm{\zeta}^1) + \DD_y(\bm{\zeta}^1) \right)\di x \\
  &\qquad +
    \int_{\Omega}\mathds{1}_{\Omega \times Y_s} \left( x,
    \frac{x}{\varepsilon} \right) \bfR \left( \frac{x}{\varepsilon}
    \right):  \DD (\ww^{\varepsilon}) : \left( \DD (\bm{\zeta}^0) + \varepsilon \DD
    (\bm{\zeta}^1) + \DD_y(\bm{\zeta}^1) \right)\di x
\end{align*}
Because $\mathds{1}_{\Omega \times Y_s} \in L^2(\Omega \times Y)$ is a Carath\'{e}odory function,
we have $\mathds{1}_{\Omega \times Y_s} \left( \cdot, \frac{\cdot}{\varepsilon} \right)$
is measurable. Notice that $\mathds{1}_{\Omega_s^{\varepsilon}} \wcv[2] \mathds{1}_{\Omega \times Y_s}$ and $\mathds{1}_{\Omega \times Y_s} \in L^2_{\per} (Y, C(\Omega))$ then, by \cref{sec:two-scale-corrector}, we have $\norm{\mathds{1}_{\Omega_s^{\varepsilon}} -
  \mathds{1}_{\Omega \times Y_s} \left( \cdot,
    \frac{\cdot}{\varepsilon} \right)}_{L^2} \cv[\varepsilon \to 0] 0.$ Thus, letting
$\varepsilon \to 0$ in the equation above and using \eqref{eq:28} lead to
\begin{align}
\label{eq:20}
  \begin{split}
    0
  &= \int_{\Omega \times Y_s} \bfR (y) : ( \DD (\vv^0) + \DD_y
    (\vv^1) ) : \left( \DD(\bm{\zeta}^0) + \DD_y (\bm{\zeta}^1) \right)
    \di x \di y\\
  &\qquad+ \int_{\Omega \times Y_s} \bfR (y) : (
    \DD(\ww^0)+ \DD_y (\ww^1)  ) : \left( \DD (\bm{\zeta}^0)  + \DD_y(\bm{\zeta}^1) \right) \di x \di y.
  \end{split}
\end{align}
The above equation characterizes $\ww^1$ on $\Omega \times Y_s$. To
uniquely determine $(\ww^0, \ww^1)$, we will show that $\DD_y (\ww^1) = 0 $ on
$\Omega \times Y_f$. Indeed, choosing $\vec{\xi} \in
C_c^{\infty}(\Omega_f^{\varepsilon})$ in \eqref{eq:r7}, we obtain 
\begin{align*}
% \label{eq:34}
\int_{\Omega} \bfB \left( \frac{x}{\varepsilon} \right): \DD(\uu^{\varepsilon}):\DD(\bm{\xi})
  \di x 
  =\int_{\Omega} \bg \cdot \bm{\xi} \di x
  - \int_{\Omega} \bfC \left( \frac{x}{\varepsilon} \right): \left(  \nabla \varphi^{\varepsilon} \otimes \nabla \varphi^{\varepsilon} \right) : \DD(\bm{\xi})\di x.
\end{align*}
Subtracting \eqref{eq:8}, and noticing that, by definition, $\varepsilon^{\gamma}
\ww^{\varepsilon} = \uu^{\varepsilon} - \vv^{\varepsilon}$, we have
\begin{align*}
\int_{\Omega} \bfB \left( \frac{x}{\varepsilon} \right) : \DD
  (\varepsilon^{\gamma} \ww^{\varepsilon}) : \DD (\vec{\xi}) \di x
  = 0, \qquad \text{ for all } \vec{\xi} \in C_c^{\infty}(\Omega_f^{\varepsilon}).
\end{align*}
By a density argument and the fact that $\bfB$ is elliptic, cf.
\eqref{eq:36}, we obtain
\begin{align*}
0 \le \varepsilon^{\gamma}\lambda_e \norm{\DD (\ww^{\varepsilon})}^2_{L^2(\Omega_f^{\varepsilon})} \le 0.
\end{align*}
We conclude that $\DD \left( \ww^{\varepsilon} \right) = 0$ in
$\Omega_f^{\varepsilon}$. This fact together with \eqref{eq:28} implies
$\DD_y(\ww^1) = 0$ on $\Omega \times Y_f$.

As in the previous step, we define the following:
\begin{itemize}
\item Cell problems:
\begin{align}
\label{eq:14}
  \begin{split}
    &\bfW^{ij} \in H_{\per}^1(Y,\RR^d)/\RR, \\
    &-\Div_y \bfR (y):  \DD_y(y_j \e^i - \bfW^{ij})  =
      0 \quad \text{ in } Y_s,\\
    &\DD_y (y_j \e^i - \bfW^{ij}) = 0 \quad
      \text{ in } Y\setminus Y_s, %to bring back \Omega \times Y in \eqref{eq:20}
    \\
    &\int_{\Gamma} \left\llbracket
      \bfR (y):  \DD_y(y_j \e^i - \bfW^{ij})
      % + 
  % \bfC (y) (\e^i + \nabla_y \chi^i)\cdot(\e^j + \nabla_y
  %     \chi^j)
      \right\rrbracket \nn\di \hmeas
  =0,\\
  &\int_{\Gamma} \left\{ \left\llbracket
    \bfR (y) : \DD_y(y_j \e^i - \bfW^{ij})
    % + 
   % \bfC \left( y \right) (\e^i + \nabla_y \chi^i)\cdot(\e^j + \nabla_y
   %  \chi^j)
    \right\rrbracket\nn \right\} \times \nn\di
  \hmeas
  =0,
  \end{split}
\end{align}
for $1 \le i,j \le d.$
\item Symmetric homogenized fourth-order tensors $\bfR^{\hom}$ and $\bfT^{\hom}$, given by
\begin{align*}
  \bfR^{\hom}_{ijmn}
  &\coloneqq \int_Y \bfR (y): \DD_y(y_j\e^i
    -\bfW^{ij}):\DD_y(y_n\e^m -\bfW^{mn}) \di y,\\
    \bfT^{\hom}_{ijmn}
  &\coloneqq \int_{Y_s} \bfR (y): \DD_y(y_j\e^i
    -\bfV^{ij}):\DD_y(y_n\e^m -\bfV^{mn}) \di y.
\end{align*}
The homogenized fourth order tensor $\bfR^{\hom}$ is symmetric, while
$\bfT^{\hom}$ may be non-symmetric.
% then \eqref{eq:20} implies 
\item The homogenized equation:
\begin{align}
\label{eq:22}
-\Div \left( \bfT^{\hom}_{ijmn} \left[ \DD(\vv^0) \right]_{ij}
  \e^m \otimes \e^n + \bfR^{\hom}_{ijmn} \left[ \DD(\ww^0) \right]_{ij}
  \e^m \otimes \e^n  \right)
  =0,
\end{align}
which, by a similar computation as above, is equivalent to 
\begin{align}
\label{eq:35}
-\Div \left( \bfT^{\hom}: \DD(\vv^0) 
  + \bfR^{\hom}: \DD(\ww^0)  \right)
  &=0.
\end{align}
\end{itemize}
\end{enumerate}
\end{proof}

\section{Conclusions}
\label{sec:conclusions}
This paper is devoted to the periodic homogenization of high-contrast dielectric elastomers. More specifically, the nonlinear system \eqref{eq:p410}-\eqref{eq:p415} of an electrostatic equation coupled with an elasticity equation with highly oscillatory periodic coefficients is considered, where $\varepsilon \ll 1$ is the characteristic size of the microstructure.  In the considered model, the elastic properties of the constituents are assumed to be vastly different and given by \eqref{eq:5}, which effectively means rigid particles. Here, we note that, for simplicity, instead of selecting the model for {\it absolutely rigid} particles, i.e. $\DD \left(\uu^{\varepsilon}\right) =0$ in $\Gamma^\varepsilon$, we chose the one of the form of \eqref{eq:5}, so we deal with an elliptic system for each fixed $\varepsilon$. The practical relevance of dielectric elastomer composites with rigid particles was mentioned in the Introduction; namely, they are used to reinforce dielectric elastomers to mitigate potential dielectric breakdown or instabilities.

The main results of this work demonstrate that the effective material behaves as a homogeneous dielectric elastomer, described by an ordered decoupled system of PDEs.  Here, `ordered' means that the solution of the electrostatics equation feeds into the first elasticity equation, which has uniformly bounded coefficients, which in turn feeds into the final equation that corresponds to the {\it high-contrast} coefficients of the original problem. The coefficients of these PDEs depend on those of the original heterogeneous material, the geometry of the composite, and the periodicity of the microstructure.  Notably, the effective coefficients are expressed in terms of the solutions to the cell problems \eqref{eq:9}, \eqref{eq:11}, and \eqref{eq:14}. The main homogenization result, provided in \cref{thm:main-results}, asserts the weak convergence of the original fine-scale problem solution to the homogenized solution in the appropriate norms, alongside the statement of the homogenized problem \eqref{eq:78}.

Previous studies, such as \cite{brianeHomogenizationNonuniformlyBounded2002,khruslovHomogenizedModelsComposite1991,marcelliniHomogenizationNonuniformlyElliptic1978}, have observed that when the coefficients of the original fine-scale problem blow up, \textit{non-local terms} emerge in the homogenized equation. In \cite{brianeHomogenizationNonuniformlyBounded2002}, conditions for the coefficients of the original fine-scale system were specified for the cases when non-local effects occur.  In the homogenized response \eqref{eq:78} obtained in this paper, we do {\it not observe} non-local phenomena from the high-contrast in the original fine-scale problem's coefficients, as the asymptotic regime of the fine-scale PDE system \eqref{eq:p410}-\eqref{eq:p415} does not align with the cases highlighted for such effects in \cite{brianeHomogenizationNonuniformlyBounded2002}.

The implications of the current investigation for practical applications in domain sciences are similar to those of any rigorous mathematical homogenization study. Specifically, by accurately capturing the interactions between different phases and the effects of microstructural periodicity, homogenization can lead to the design of materials with tailored properties. Furthermore, homogenization techniques can optimize the use of composite materials by providing effective material properties that simplify complex microstructures into manageable macroscopic models, reducing computational costs and improving the efficiency of simulations used in the design and testing of new materials.

%, unlike the results obtained in \cite{tianDielectricElastomerComposites2012,francfortEnhancementElastodielectricsHomogenization2021}.
\appendix
\section*{Acknowledgements}
The work of the third author was supported  by NSF grant DMS-2110036.  This
material is based upon work supported by and while serving at the
National Science Foundation for the second author Yuliya Gorb. Any
opinion, findings, and conclusions or recommendations expressed in
this material are those of the authors and do not necessarily reflect
views of the National Science Foundation.

%

% \section{Extension to the nonlinear case}
% \label{sec:extens-nonl-case}

%\section{Appendix}
%\label{sec:an-appendix}

\bibliographystyle{siamplain}%{habbrv}
%\bibliography{homogenisation}
\bibliography{homogenisation,arxiv}%,topological-insulator}

\end{document}